\documentclass[a4paper,12pt,reqno]{amsart}
\usepackage{amssymb,amsmath,array,amscd,amsthm,hhline}

\usepackage[mathscr]{euscript}
\usepackage{stmaryrd}
\usepackage{ulem}

\usepackage{tikz-cd}

\usepackage{mathrsfs}

\usepackage{xcolor}

\def\({\left(}
\def\){\right)}

\def\lm{\lambda}

\def\R{\mathbb R}

\def\Seq{\mathbf{Seq}}

\def\<{\langle}
\def\>{\rangle}

\renewcommand\phi{\varphi}
\def\bbeta{{\boldsymbol\beta}}
\def\ppi{{\boldsymbol\pi}}

\renewcommand\epsilon{\varepsilon}

\def\suchthat{\mathbin{\rm |}}

\def\bbeta{{\boldsymbol\beta}}

\newtheorem{theorem}{Theorem}
\newtheorem{proposition}[theorem]{Proposition}
\newtheorem{lemma}[theorem]{Lemma}

\newtheorem{remark}[theorem]{Remark}

\newtheorem{definition}[theorem]{Defintition}
\newtheorem{example}[theorem]{Example}

\def\Gal{\mathbf{Gal}}
\def\gal{\mathop{\mathrm{gal}}}

\def\f{\mathbf{f}}

\def\le{\leqslant}
\def\ge{\geqslant}

\renewcommand{\labelenumi}{{\rm\theenumi}}
\renewcommand{\theenumi}{{\rm(\arabic{enumi})}}

\def\={\equiv}

\def\Ch{\mathbf{Ch}}

\def\BS{\mathrm{BS}}

\def\GL{\mathrm{GL}}

\title{Galleries for root subsystems}

\author{Vladimir Shchigolev}

\address{Financial University under the Government of the Russian Federation, 49 Leningradsky Prospekt, Moscow, Russia}
\email{shchigolev\_vladimir@yahoo.com}

\begin{document}

\begin{abstract} 
We consider projection and lifting of labelled galleries to and from roots subsystems. 
Our constructions allow us to construct some topological embeddings of Bott-Samelson varieties
skew equivariant with respect to the compact torus and order-preserving on the sets of points
fixed by it.
\end{abstract}

\maketitle

\section{Introduction}

Let $E$ be a Euclidian space and $\Phi\subset E$ be a finite root system. We require neither that $\Phi$ span $E$ 
nor that $\Phi$ be crystallographic. A subset $\Psi\subset\Phi$ is called a {\it root subsystem} if 
it is not empty and is stable under the reflections with respect to the roots $\alpha\in\Psi$ 
(in other words, if $\Psi$ is a root system in its own right). 
Root subsystems were objects of extensive study, see for example~\cite{Wallach} and~\cite{Oshima}.
Moreover A.\;Borel and J. de Siebental~\cite{Borel_de_Siebental} established a connection between closed connected subgroups 
of compact Lie subgroups of maximal rank and the maximal closed root subsystems.
The aim of this paper is to relate  
chambers and galleries for the root systems 
$\Phi$ and $\Psi$ 
and apply these results to Bott-Samelson varieties.

For any root $\alpha\in\Phi$, we denote by $L_\alpha$ the hyperplane in $E$ perpendicular to $\alpha$.
We consider $E$ with respect to the metric topology and for any subset $S\subset E$ we denote by $\bar S$ its closure.
For a subset $X\subset\Phi$, we denote the connected components of the difference
$E\setminus\bigcup_{\alpha\in X}L_\alpha$ by $\Ch_X$ and call them {\it chambers} with respect to $X$. 
We say that a chamber $C\in\Ch_X$ is {\it attached} to $L_\alpha$ if $\overline C\cap L_\alpha$ has a nonempty interior
in $L_\alpha$. Similarly, two chambers $C,D\in\Ch_X$ are {\it connected} through $L_\alpha$ if
$\overline C\cap\overline D\cap L_\alpha$ has a nonempty interior in $L_\alpha$. 

For any chamber $C\in\Ch_\Phi$, we denote by $C_\Psi$ the unique chamber of $\Ch_\Psi$ containing $C$. 
Proposition~\ref{proposition:1}, allows us to make a (labelled) gallery of chambers of $\Ch_\Psi$ 
from a gallery of chambers of $\Ch_\Phi$ in the obvious way, see Section~\ref{Projection}.
We call this procedure {\it projection}.

The transition in the opposite direction, called here {\it lifting}, is more difficult and is not unique.
First we have to lift chambers of $\Ch_\Psi$ to chambers of $\Ch_\Phi$, see Theorem~\ref{theorem:1}.
Then we prove Theorem~\ref{theorem:2} for galleries. This result establishes 
existence and uniqueness in a weak form. 
 
The projection and lifting procedures applied successively allow us to construct the so-called $p$-pairs of galleries
described by some wall crossing properties, see Definition~\ref{def:1}. We also
discuss gluing of galleries produced this way, see Section~\ref{Gluing}. 
 
Finally, we embed Bott-Somelson varieties into one another in Section~\ref{Bott-Samelson varieties}
as an application of the theory we develop.  
Our main result here is Theorem~\ref{theorem:4},
which asserts the existence of topological embeddings of Bott-Samelson varieties 
skew equivariant with respect to the compact torus $K$ and order-preserving on the set of the $K$-fixed points.

\section{Roots subsystems}\label{Roots_subsystems}

\subsection{Definitions and notation} We assume in the sequel the notation of the introduction. 
We denote by $\omega_\alpha$ the reflection of $E$ through $L_\alpha$. 
The Weyl group $W$ is the subgroup of $\GL(E)$ generated by all reflections $\omega_\alpha$. 
We denote its neutral element by $1$. The scalar product on $E$ is denoted by $(\cdot,\cdot)$.
For any subset $X\subset E$, we set $-X=\{-x\suchthat x\in X\}$ and denote by $\R X$ the set of all $\R$-linear combinations 
$c_1x_1+\cdots+c_nx_n$, where $c_1,\ldots,c_n\in\R$ and $x_1,\ldots,x_n\in X$.
This linear combination is called {\it positive} (respectively, {\it negative}) if $c_1,\ldots,c_n>0$ (respectively, $c_1,\ldots,c_n<0$).

We will generally follow~\cite{Steinberg} and recall now briefly the main definitions introduced therein.
An {\it ordering} of $E$ is a subset $E^+\subset E$ closed under addition and multiplication 
by positive scalars such that $E^+\cap(-E^+)=0$ and $E^+\cup(-E^+)=E$.
The {\it positive} (respectively, {\it negative}) {\it system} of $\Phi$ relative to this ordering is the intersection $\Phi^+=\Phi\cap E^+$
(respectively, $\Phi^-=\Phi\cap(-E^+)$). The {\it fundamental chamber} of this positive system is
$$
\{e\in E\suchthat \forall \alpha\in\Phi^+\;(e,\alpha)>0\}.
$$
A {\it simple system} of $\Phi$ is a subset $X\subset\Phi$ that is linearly independent
and any root of $\Phi$ is either a positive or a negative linear combination of elements of $X$.  

For any chamber $C\in\Ch_\Phi$, we denote by $\Phi^+(C)$ the positive system with fundamental chamber $C$, that is,
$$
\Phi^+(C)=\{\alpha\in\Phi\suchthat\forall e\in C\; (e,\alpha)>0\}.
$$
Clearly, $C$ is the fundamental chamber of $\Phi^+(C)$.

We also define $\Phi^-(C)=-\Phi^+(C)$ and choose the (unique) simple system $\Phi^s(C)$ contained in $\Phi^+(C)$.
These operators are $W$-invariant:
\begin{equation}\label{eq:10}
\Phi^+(wC)=w\Phi^+(C),\quad \Phi^-(wC)=w\Phi^-(C),\quad \Phi^s(wC)=w\Phi^s(C)
\end{equation}
for any $w\in W$. We will use the similar notation $\Psi^+(C)$, $\Psi^-(C)$ and $\Psi^s(C)$ for $C\in\Ch_\Psi$.

The proof of the following simple result is left to the reader.
\begin{proposition}\label{proposition:1} (i) Let $C$ and $D$ be chambers of $\Ch_\Phi$ connected through $L_\beta$ for $\beta\in\Phi\setminus\Psi$.
Then $C_\Psi=D_\Psi$.

\vspace{2pt}

(ii) Let $C$ and $D$ be chambers of $\Ch_\Phi$ connected through $L_\alpha$ for $\alpha\in\Psi$. Then $C_\Psi$ and $D_\Psi$
are also connected through $L_\alpha$. Moreover, $C_\Psi=D_\Psi$ if and only if $C=D$.
\end{proposition}

Our favorite root subsystems are given by the following result.

\begin{lemma}\label{lemma:0} Let $\Psi$ be a nonempty subset of $\Phi$.
The following conditions are equivalent:
{\renewcommand{\labelenumi}{{\rm\theenumi}}
\renewcommand{\theenumi}{{\rm(\roman{enumi})}}
\begin{enumerate}
\item\label{lemma:0:p:1} $\Psi=\R X\cap\Phi$ for some nonempty
                         subset $X\subset\Phi$;\\[-10pt]
\item\label{lemma:0:p:2} $\Psi=\R\Psi\cap\Phi$.
\end{enumerate}}
\noindent
A subset $\Psi$ satisfying these conditions is a root subsystem of $\Phi$.
\end{lemma}
\begin{proof} Clearly~\ref{lemma:0:p:2} implies~\ref{lemma:0:p:1}.
Suppose conversely that~\ref{lemma:0:p:1} is satisfied. As $\Psi\subset\R X$, we get $\R\Psi\subset\R X$.
On the other hand, $X\subset\Psi$, whence $\R X\subset\R\Psi$. Thus $\R X=\R\Psi$ and we get~\ref{lemma:0:p:2}.
The last claim holds, as both $\R\Psi$ and $\Phi$ are stable under the reflections with respect to the roots of $\Psi$.
\end{proof} 
\noindent
If the conditions of this lemma hold, we say that $\Psi$ is {\it saturated} in $\Phi$.

\subsection{Lifting}\label{Lifting} Let $C\in\Ch_\Phi$ and $D\in\Ch_\Psi$.
We say that $C$ {\it lifts} $D$ if $\Psi^s(D)\subset\Phi^s(C)$.
Obviously, the last inclusion implies $\Psi^+(D)\subset\Phi^+(C)$.
The converse is not true, see Example~\ref{example:1}.
%

\begin{lemma}\label{lemma:1}
If $C$ lifts $D$, then $C_\Psi=D$.
\end{lemma}
\begin{proof} As $C$ and $D$ are the fundamental chambers of $\Phi^+(C)$ and $\Psi^+(D)$, respectively 
and $\Psi^+(D)\subset\Phi^+(C)$, we get $C\subset D$. Hence by definition $D=C_\Psi$.
%
\end{proof}

\begin{theorem}\label{theorem:1} (i) 
Suppose that $\Psi$ is saturated in $\Phi$.
Then for any $D\in\Ch_\Psi$, 
there exists at least one chamber $C\in\Ch_\Phi$ lifting~$D$.

\!\!\!(ii) Suppose that a chamber $C\in\Ch_\Phi$ lifts a chamber $D\in\Ch_\Psi$. 
Then $\Psi$ is saturated in~$\Phi$.
\end{theorem}
\begin{proof} (i) We denote $V=\R\Psi$ for brevity. Let $V^+$ be an ordering of $V$ such that 
$\Psi^+(D)=\Psi^+(D\cap V)=\Psi\cap V^+$. Let $e_1,\ldots,e_k$ be a basis of the orthogonal compliment 
$V^\perp=\{e\in E\suchthat\forall v\in V(e,v)=0\}$.
We define $E^+$ as the set of all sums $c_1e_1+\cdots+c_ke_k+v$,
where $c_1,\ldots,c_k\in\R$ and either $c_1=\cdots=c_k=0$ and $v\in V^+$
or $v\in V$ and there is some $i=1,\ldots,k$ such that $c_1=\cdots=c_{i-1}=0$ and $c_i>0$.
It is clear that $E^+$ is an ordering of $E$ and $V^+=E^+\cap V$.
Moreover, the following property holds:

\begin{itemize}
\item[($*$)] A positive linear combination of vectors of $E^+$ belongs to $V$ 
if and only if each vector belongs to $V^+$.
\end{itemize}

Let $C$ be the fundamental chamber relative to the positive system $\Phi\cap E^+$.
Obviously $\Psi^+(D)\subset\Phi^+(C)=\Phi\cap E^+$. We claim also that $\Psi^s(D)\subset\Phi^s(C)$. 
Indeed let $\alpha\in\Psi^s(D)$. Then $\alpha\in\Phi^+(C)$ and we have a decomposition
$\alpha=c_1\beta_1+\cdots+c_n\beta_n$, where $\beta_1,\ldots,\beta_n\in\Phi^s(C)$ 
and $c_1,\ldots,c_n>0$. By~($*$), we get $\beta_1,\ldots,\beta_n\in V^+$.
As $\Psi$ is saturated in $\Phi$, we also get $\beta_1,\ldots,\beta_n\in\Psi$, whence $\beta_1,\ldots,\beta_n\in\Psi^+(D)$.
We get a contradiction unless $n=1$ and $c_1=1$, which proves the claim. 

(ii) 
Let $\Phi^s(C)=\{\alpha_1,\ldots,\alpha_n\}$, where $\alpha_1,\ldots,\alpha_k\in\Psi^s(D)$, 
$\alpha_{k+1},\ldots,\alpha_n\notin\Psi^s(D)$ and all roots $\alpha_1,\ldots,\alpha_n$ are pairwise distinct.

Let $\beta\in\R\Psi\cap\Phi^+(C)$. 
There are two $\R$-linear decompositions
$\beta=c_1\alpha_1+\cdots+c_k\alpha_k$ and $\beta=d_1\alpha_1+\cdots+d_n\alpha_n$ with $d_1,\ldots,d_n>0$.
Hence $d_{k+1}=\cdots=d_n=0$ and $\beta$ is a positive linear combination of the roots of $\Psi^s(D)$.
Similarly, if $\beta\in\R\Psi\cap\Phi^-(C)$, then $\beta$ is a negative combination of the roots of $\Psi^s(D)$.
Thus $\Psi^s(D)$ is a simple system for both root systems $\R\Psi\cap\Phi$ and $\Psi$.
As a simple system uniquely determines its root system, we get $\R\Psi\cap\Phi=\Psi$.
\end{proof}

\begin{example}\label{example:1}
\rm Let $\Phi$ be the root system of type $A_2$ with the simple roots $\alpha$ and $\beta$.
We consider the root subsystem $\Psi=\{-\alpha,\alpha\}$.
For any root $\gamma\in\Phi$ (respectively $\gamma\in\Psi$), we denote by $C(\gamma)$ (respectively $D(\gamma)$) 
the chamber of $\Ch_\Phi$ (respectively $\Ch_\Psi$) containing $\gamma$. 
The only chambers of $\Ch_\Phi$ lifting $D(\alpha)$ are $C(\alpha+\beta)$ and $C(-\beta)$.
On the other hand,
$\Psi^+(D(\alpha))=\{\alpha\}\subset\{-\beta,\alpha,\alpha+\beta\}=\Phi^+(C(\alpha))$ and 
$\Psi^s(D(\alpha))=\{\alpha\}\not\subset\{-\beta,\alpha+\beta\}=\Phi^s(C(\alpha))$.
\end{example}

\begin{example}\label{example:2}
\rm Let $\Phi$ be the root system of type $B_2$ and $\Psi$ be its subsystem consisting of the short roots (of type $A_1\times A_1$).
One root of any simple system 
of $\Phi$ is long and both roots of any simple system 
of $\Psi$ 
are short. Therefore, no chamber of $\Ch_\Phi$ lifts a chamber of $\Ch_\Psi$. We have $\R\Psi\cap\Phi\ne\Psi$,
that is, $\Psi$ is not saturated in $\Phi$ 
in accordance with Theorem~\ref{theorem:1}.
\end{example}

We will need the following additional criterion.

\begin{lemma}\label{lemma:2}
Let $C,D\in\Ch_\Phi$ and $\Psi$ be a root subsystem of $\Phi$.
The chamber $C$ lifts $D_\Psi$ if and only if there exists
$X\subset\Phi^s(C)\cap\Phi^+(D)$ such that $\Psi=\R X\cap\Phi$.
If these conditions hold, then we can take $X=\Psi^s(D_\Psi)$.
\end{lemma}
\begin{proof}
Suppose that $C$ lifts $D_\Psi$. Let us take $X=\Psi^s(D_\Psi)$. 
By our assumption, we have $X\subset\Phi^s(C)$. 
On the other hand, $X\subset\Psi^+(D_\Psi)$ and $D\subset D_\Psi$.
Hence for any $\alpha\in X$ and $e\in D$, we get $(\alpha,e)>0$. Thus $\alpha\in\Phi^+(D)$.
As $\alpha$ was chosen arbitrarily in $X$, we get $X\subset\Phi^+(D)$.
Moreover, by Theorem~\ref{theorem:1}(ii), we get
$\R X\cap\Phi=\R\Psi\cap\Phi=\Psi$.

Conversely, suppose that a set $X$ as in the formulation of the lemma exists.
%
We claim that $X$ is a simple system for $\Psi$. Indeed let $\Phi^s(C)=\{\alpha_1,\ldots,\alpha_n\}$, 
where $\alpha_1,\ldots,\alpha_k\in X$,
$\alpha_{k+1},\ldots,\alpha_n\notin X$ and all roots $\alpha_1,\ldots,\alpha_n$ are pairwise distinct.
Let $\beta\in\Psi\cap\Phi^+(C)$. There are two $\R$-linear decompositions
$\beta=c_1\alpha_1+\cdots+c_k\alpha_k$ and $\beta=d_1\alpha_1+\cdots+d_n\alpha_n$ with $d_1,\ldots,d_n>0$.
Hence $d_{k+1}=\cdots=d_n=0$ and we have the positive linear combination $\beta=d_1\alpha_1+\cdots+d_k\alpha_k$.
We get a similar negative linear combination for $\beta\in\Psi\cap\Phi^-(C)$. Finally note that the set $X$
is linearly independent as a subset of the simple system $\Phi^s(C)$.

Let $\alpha\in X$, $e\in D$ and $v\in D_\Psi$. As $\alpha\in\Psi$, the hyperplane $L_\alpha$ does not separate
the points $e$ and $v$. Hence and from $(e,\alpha)>0$, we get $(v,\alpha)>0$.
Thus we have proved that $\alpha\in\Psi^+(D_\Psi)$. As $\alpha$ was chosen arbitrarily in $X$,
we get $X\subset\Psi^+(D_\Psi)$. Hence and from the previous paragraph we get $\Psi^s(D_\Psi)=X\subset\Phi^s(C)$.
\end{proof}

\section{Galleries}

\subsection{Definitions} For any integers $i\le j$, we denote $[i,j]=\{i,i+1,\ldots,j\}$, $[i,j)=\{i,i+1,\ldots,j-1\}$, etc. 
A {\it gallery}\footnote{It would be more correct to use the term ``labelled gallery''. However we omit the adjective,
as we do not consider here unlabelled galleries.  On the other hand, we prefer to keep the adjective, 
when we consider combinatorial galleries in Section~\ref{Combinatorial galleries}.} in $\Phi$ is a sequence
\begin{equation}\label{eq:i}
\Gamma=(C_0,L_{\alpha_1},C_1,L_{\alpha_2},\ldots,C_{n-1},L_{\alpha_n},C_n),
\end{equation}
where $C_0,C_1,\ldots,C_n\in\Ch_\Phi$ and $\alpha_1,\ldots,\alpha_n\in\Phi$ are such that 
the chambers $C_{i-1}$ and $C_i$ are connected through $L_{\alpha_i}$ for any $i=1,\ldots,n$.
Clearly, either $C_i=\omega_{\alpha_i}C_{i-1}$ or $C_i=C_{i-1}$.
The number $n$ of walls in $\Gamma$ is called the {\it length} of $\Gamma$ and is denoted by $|\Gamma|$.
The chambers $C_0$ and $C_n$ are called the {\it initial} and the {\it final} chambers of $\Gamma$ respectively
and the chamber $C_i$ is called the {\it $i$th} chamber of $\Gamma$. We also say that $\Gamma$ {\it begins} at $C_0$
and {\it ends} at $C_n$.

For $\Gamma$ given by~(\ref{eq:i}) and
integers $i$ and $j$ such that $0\le i\le j\le n$, we consider the
following gallery:
$$
\Gamma_{[i,j]}=(C_i,L_{\alpha_{i+1}},C_{i+1},L_{\alpha_{i+2}},\ldots,C_{j-1},L_{\alpha_j},C_j).
$$  
Moreover, if the final chamber of a gallery $\Gamma$ coincides with the initial chamber of a gallery $\Delta$, 
then they can by glued
$$
\Gamma\cup\Delta=(C_0,L_{\alpha_1},\ldots,C_{n-1},L_{\alpha_n},C_n,L_{\beta_1},\ldots,D_{m-1},L_{\beta_m},D_m),
$$
where $\Gamma$ is denoted as in~(\ref{eq:i}) and $\Delta=(C_n,L_{\beta_1},D_1,L_{\beta_2},\ldots,D_{m-1},L_{\beta_m},D_m\}$.

We will pay special attention to the {\it sequence of walls} $(L_{\alpha_1},L_{\alpha_2}\ldots,L_{\alpha_n})$ 
of a gallery $\Gamma$ denoted as in~(\ref{eq:i}). 
Considering its subsequences, we get the following definition that is very
close to the definition of a $(p,w)$-pair from~\cite{fcat} (see Section 3.2 therein or Section~\ref{(p,w)-pairs} in this paper).

\begin{definition}\label{def:1} Let $\Gamma=(C_0,L_{\alpha_1},\ldots,C_{n-1},L_{\alpha_n},C_n)$ and
$\Delta=(D_0,L_{\beta_1},\ldots,D_{m-1},L_{\beta_m},D_m)$
be galleries and $p:[1,m]\to[1,n]$ be an increasing embedding.
We call $(\Delta,\Gamma)$ a $p$-pair iff $L_{\alpha_{p(i)}}=L_{\beta_i}$
for any $i=1,\ldots,m$.
\end{definition}  
\noindent
For any pair of galleries $(\Gamma,\Delta)$ as in this definition, we define its {\it sign} $\epsilon=(\epsilon_1,\ldots,\epsilon_m)$ by 
$$
\epsilon_i=\left\{\!\!
\begin{array}{rl}
1&\text{ if }L_{\beta_i}\text{does not separate }D_i\text{ and }C_{p(i)};\\
-1&\text{otherwise}.
\end{array}
\right.
$$
Similarly, we define the {\it cosign} $\mu=(\mu_1,\ldots,\mu_m)$ by
$$
\mu_i=\left\{\!\!
\begin{array}{rl}
1&\text{ if }L_{\beta_i}\text{does not separate }C_{i-1}\text{ and }D_{p(i)-1};\\
-1&\text{otherwise}.
\end{array}
\right.
$$
We say that the sign or cosign is {\it positive} if all its entries are so.

\subsection{Projection}\label{Projection}
Let $\Psi$ be a root subsystem of $\Phi$. For a gallery $\Gamma$ in $\Phi$ denoted as in~(\ref{eq:i}), 
we denote
$$
I_\Psi(\Gamma)=\{i\in[1,n]\suchthat\alpha_i\in\Psi\}=\{i_1<i_2<\cdots<i_m\}.
$$
We define
$$
\Gamma_\Psi=((C_0)_\Psi,L_{\alpha_{i_1}},(C_{i_1})_\Psi,L_{\alpha_{i_2}},(C_{i_2})_\Psi,\ldots,(C_{i_{m-1}})_\Psi,L_{\alpha_{i_m}},(C_{i_m})_\Psi).
$$
Note that for any $t=1,\ldots,m+1$ and $j\in[i_{t-1},i_t)$, we get
\begin{equation}\label{eq:iii}
(C_{i_{t-1}})_\Psi=(C_j)_\Psi.
\end{equation}
Here and in what follows, we assume $i_0=0$.
From Proposition~\ref{proposition:1}, we obtain the following result. 
\begin{proposition}\label{proposition:2}
For any gallery $\Gamma$ in $\Phi$ and any root subsystem $\Psi$ of $\Phi$, 
the sequence $\Gamma_\Psi$ is a gallery in $\Psi$.
\end{proposition}

\subsection{Lifting}\label{Lifting+} We formulate the following definition inspired by Section~\ref{Lifting}. 
\begin{definition} A gallery $\Gamma$ in $\Phi$ lifts a gallery $\Delta$ in $\Psi$
if the sequences of walls of these galleries coincide and the $i$th chamber of $\Gamma$ lifts the i$th$ chamber 
of $\Delta$ for any admissible $i$. 
\end{definition}

\noindent
We are going to prove the following existence and uniqueness result.

\begin{theorem}\label{theorem:2} 
For any gallery $\Delta$ in $\Psi$ and a chamber $C\in\Ch_\Phi$ lifting the $i$th chamber of $\Delta$,
there exists a unique gallery in $\Phi$ lifting $\Delta$ whose $i$th chamber is $C$.
\end{theorem}
\begin{proof}
{\it Existence.} Let us denote
\begin{equation}\label{eq:ii}
\Delta=(D_0,L_{\beta_1},D_1,\ldots,L_{\beta_m},D_m)
\end{equation}
We set $C_i=C$, define the chambers $C_1,\ldots,C_{i-1}$ recursively (from right to left) by
$$
C_{j-1}=\left\{\!
\begin{array}{ll}
C_j&\text{if }D_{j-1}=D_j;\\
\omega_{\beta_j}C_j&\text{otherwise}.
\end{array}
\right.
$$
and the chambers $C_{i+1},\ldots,C_m$ recursively (from left to right) by
$$
C_{j+1}=\left\{\!
\begin{array}{ll}
C_j&\text{if }D_{j+1}=D_j;\\
\omega_{\beta_{j+1}}C_j&\text{otherwise}.
\end{array}
\right..
$$
Then we define $\Delta^C_i=(C_0,L_{\beta_1},C_1,\ldots,L_{\beta_m},C_m)$ and claim that it is the required gallery.

First we consider the case $i=0$. We will apply induction on $m$. 
The case $m=0$ is obvious. 
Suppose that $m>0$ and that the claim is true for the
truncated gallery $\widetilde\Delta=(D_0,L_{\beta_1},D_1,\ldots,L_{\beta_{m-1}},D_{m-1})$.
As $D_{m-1}$ is attached to $L_{\beta_m}$, we get
$\beta_m\in\Psi^s(D_{m-1})\subset\Phi^s(C_{m-1})$. This proves that $C_{m-1}$
is attached to $L_{\beta_m}$. If $D_m=D_{m-1}$, then by definition $C_m=C_{m-1}$
and we are done. If $D_m\ne D_{m-1}$, then $D_m=\omega_{\beta_m}D_{m-1}$
and by definition $C_m=\omega_{\beta_m}C_{m-1}$.
Hence $\Psi^s(D_m)=\omega_{\beta_m}\Psi^s(D_{m-1})\subset\omega_{\beta_m}\Phi^s(C_{m-1})=\Phi^s(C_m)$
by~(\ref{eq:10}).

The case $i=m$ can be treated similarly. In the case $0<i<m$, we get $\Delta^C_i=(\Delta_{[1,i]})^C_i\cup(\Delta_{[i,m]})_0^C$
and the result follows from the two cases considered above.

{\it Uniqueness.} Let $\Gamma=(C'_1,L_{\beta_1},\ldots,C'_{m-1},L_{\beta_m},C'_m)$ be a gallery in $\Phi$ that lifts $\Delta$ whose $i$th chamber is $C$.
We are going to prove that $\Gamma=\Delta_i^C$. Here we use the notation of the previous part, that is, 
denote $\Delta$ as in~(\ref{eq:ii}) and compute recursively the chambers $C_1,\ldots,C_m$.
Suppose that we have already proved that $C_j=C'_j$ for some $j=i,\ldots,m-1$.
 
First consider the case $D_{j+1}=D_j$. Suppose that $C'_{j+1}\ne C_{j+1}=C_j$. Then $C'_{j+1}=\omega_{\beta_{j+1}}C_j$.
Changing the sign of $\beta_{j+1}$, if necessary, we can assume that $\beta_{j+1}\in\Psi^s(D_j)$.
As $C'_{j+1}$ lifts $D_{j+1}$ and $C_j$ lifts $D_j$, we get $\beta_{j+1}\in\Phi^s(C_j)\cap\Phi^s(C'_{j+1})$. 
On the other hand, $-\beta_{j+1}=\omega_{\beta_{j+1}}\beta_{j+1}\in\omega_{\beta_{j+1}}\Phi^s(C_j)=\Phi^s(C'_{j+1})$,
which is a contradiction. Thus we have proved that $C'_{j+1}=C_{j+1}$.

Finally, we consider the case $D_{j+1}\ne D_j$. In this case $D_{j+1}=\omega_{\beta_{j+1}}D_j$. 
Suppose that $C'_{j+1}\ne C_{j+1}=\omega_{\beta_{j+1}}C_j$. In that case $C'_{j+1}=C_j$.
Changing the sign of $\beta_{j+1}$, if necessary, we can assume that $\beta_{j+1}\in\Psi^s(D_j)\subset\Phi^s(C_j)$.
Then we get $-\beta_{j+1}=\omega_{\beta_{j+1}}\beta_{j+1}\in\omega_{\beta_{j+1}}\Psi^s(D_j)=\Psi^s(D_{j+1})\subset\Psi^s(C'_{j+1})$.
Therefore, we again have a contradiction $\beta_{j+1},-\beta_{j+1}\in \Phi^s(C_j)$. 
Thus we have proved that $C'_{j+1}=C_{j+1}$. 

Arguing inductively with the help of the previous two paragraphs, we can prove that $C'_j=C_j$ for any $j=i,\ldots,m$.
The prove of this equality for $j=1,\ldots,i$ is similar.
\end{proof}

\begin{remark} \rm
We have proved that any lifting of $\Delta$ has the form $\Delta_i^C$. 
We will use this notation throughout the paper. Moreover, we abbreviate ${}^C\!\Delta=\Delta^C_0$ and $\Delta^C=\Delta^C_{|C|}$. 
\end{remark}

\begin{remark} \rm
It follows from Lemma~\ref{lemma:1} that $\Gamma_\Psi=\Delta$ if $\Gamma$ lifts $\Delta$.
\end{remark}

\subsection{Constructing $p$-pairs} Let $\Gamma$ be a gallery in $\Phi$.
Choosing a saturated root subsystem $\Psi$ of $\Phi$, we get the gallery
$\Gamma_\Psi$. Now choosing a chamber $C\in\Ch_\Phi$ that lifts the $i$th chamber
of $\Gamma_\Psi$, we get the gallery $(\Gamma_\Psi)^C_i$.

\begin{theorem}\label{theorem:3} For any gallery $\Gamma$ in $\Phi$
and a gallery $\Delta$ in $\Phi$ lifting $\Gamma_\Psi$, the pair $(\Delta,\Gamma)$ is a $p$-pair of positive sign and cosign,
where $p:[1,|\Gamma_\Psi|]\to[1,|\Gamma|]$ is the increasing embedding with image $I_\Psi(\Gamma)$.
\end{theorem}
\begin{proof}
Let $\Gamma=(C'_0,L_{\alpha_1},C'_1,L_{\alpha_2},\ldots,C'_{n-1},L_{\alpha_n},C'_n)$ 
and $I_\Psi(\Gamma)=\{i_1<\cdots<i_m\}$. We have $p(t)=i_t$. 
We $\Delta=(C_0,L_{\alpha_{i_1}},C_1,\ldots,L_{\alpha_{i_m}},C_m)$
for some chambers $C_0,\ldots,C_m\in\Ch_\Phi$. Hence we derive the claim about being a $p$-pair.

Let us prove the positivity claim. Let $\mu=(\mu_1,\ldots,\mu_m)$ be the cosign of our pair. 
As usual, we set $i_0=0$. Let us consider any $t=1,\ldots,m$. The chambers $(C'_{i_{t-1}})_\Psi$ and $C_{t-1}$
are the $t-1$th chambers of the galleries $\Gamma_\Psi$ and $\Delta$ respectively.
By definition, the latter chamber lifts the former one. 
Hence by Lemma~\ref{lemma:1} and~(\ref{eq:iii}) (with primes added), 
we get $(C_{t-1})_\Psi=(C'_{i_{t-1}})_\Psi=(C'_{i_t-1})_\Psi$.
Therefore $L_{\alpha_{i_t}}$ does not separate $C_{t-1}$ and $C'_{i_t-1}$, whence $\mu_t=1$. 
Thus we have proved that the cosign is positive. The positivity of the sign is proved similarly.
%
%
\end{proof}

\subsection{Gluing}\label{Gluing} Let $\Gamma$ be a gallery denoted as in~(\ref{eq:i}).
Suppose additionally that we are given integers $i_0,\ldots,i_k$ such that
$0=i_0\le i_1\le\cdots\le i_k=n$. We choose a number $q=0,\ldots,k-1$ and an arbitrary
saturated root subsystem $\Psi_q$ of $\Phi$. Then we take any gallery $\Delta_q$ lifting $(\Gamma_{[i_q,i_{q+1}]})_{\Psi_q}$. 
If $q<k-1$, then we choose any nonempty subset $X\subset\Phi^s(D)\cap\Phi^+(C_{i_{q+1}})$, 
where $D$ is the final chamber of $\Delta_q$, and consider the saturated root subsystem $\Psi_{q+1}=\R X\cap\Phi$.
Lemma~\ref{lemma:2} guarantees that $D$ lifts $(C_{i_{q+1}})_{\Psi_{q+1}}$. 
Then we define $\Delta_{q+1}={}^D(\Gamma_{i_{q+1},i_{q+2}})_{\Psi_{q+1}}$.
Similarly, if $0<q$, then considering the initial chamber $F$ of $\Delta_q$, 
we can choose a root subsystem $\Psi_{q-1}=\R Y\cap\Phi$ for a nonempty $Y\subset\Phi^s(F)\cap\Phi^+(C_{i_q})$ 
and define the gallery $\Delta_{q-1}=((\Gamma_{i_{q-1},i_q})_{\Psi_{q-1}})^F$. 
Continuing to the left and to the right, we will construct the galleries 
$\Delta_0,\ldots,\Delta_{q-1},\Delta_{q+1},\ldots,\Delta_{k-1}$.
It remains to glue them all to get the gallery $\Delta_0\cup\Delta_1\cup\cdots\cup\Delta_{k-1}$.

\begin{remark} \rm
In the above construction, we could always take $X\subset\Psi_q^s((C_{i_{q+1}})_{\Psi_q})$.
If $X=\Psi_q^s((C_{i_{q+1}})_{\Psi_q})$, then we get $\Psi_{q+1}=\Psi_q$.
The similar argument applies to $Y$. If we always make the choices 
with the equalities just mentioned, then $\Delta$ will lift the gallery $\Gamma_{\Psi_q}$.
\end{remark}

By Theorem~\ref{theorem:3}, each pair $(\Delta_l,\Gamma_{[i_l,i_{l+1}]})$ is a $p_l$-pair for the corresponding 
embedding $p_l:[1,|\Delta_l|]\to[1,|\Gamma_{[i_l,i_{l+1}]}|]$. We can glue the embeddings $p_0,\ldots,p_{k-1}$ 
to the embedding $$
p:[1,|\Delta_0|+|\Delta_1|+\cdots+|\Delta_{k-1}|]\to[1,|\Gamma|]
$$ by 
$p(j)=p_l(j-|\Delta_0|-\cdots-|\Delta_{l-1}|)+i_l$ for $|\Delta_0|+\cdots+|\Delta_{l-1}|<j\le|\Delta_0|+\cdots+|\Delta_l|$,
where $i_0=0$.

\begin{proposition}
The pair $(\Delta_0\cup\Delta_1\cup\cdots\cup\Delta_{k-1},\Gamma)$ is a $p$-pair of positive sign and cosign. 
\end{proposition}

\begin{example}\rm Let us consider examples where projection and lifting applied successively
produce the same gallery. One obvious case is $\Psi=\Phi$.
The other example is the following.
Let $\Gamma$ be a gallery denoted as in~(\ref{eq:i}) We write $\Gamma$ indicating the segments for each root subsystem
as follows:
$$
\Delta=(\underbrace{C_0,L_{\alpha_1},C_1}_{\{\alpha_1,-\alpha_1\}}\!\!\!\!\!\!\!\!\overbrace{\quad\;,L_{\alpha_2},C_2}^{\{\alpha_2,-\alpha_2\}},\ldots,\underbrace{C_{n-1},L_{\alpha_n},C_n}_{\{\alpha_n,-\alpha_n\}}).
$$
Then we get the decomposition
$$
\Gamma={}^{C_0}(\Gamma_{[0,1]})_{\{\alpha_1,-\alpha_1\}}\cup{}^{C_1}(\Gamma_{[1,2]})_{\{\alpha_2,-\alpha_2\}}\cup\cdots\cup{}^{C_{n-1}}(\Gamma_{[n-1,n]})_{\{\alpha_n,-\alpha_n\}}.
$$
\end{example}

\section{Bott-Samelson varieties}\label{Bott-Samelson varieties}

\subsection{Construction}
Let $G$ be a semisimple complex algebraic group.
We fix a Borel subgroup $B<G$ and a maximal torus $T<B$. We denote by $W$ the Weyl group of $G$ 
and by $\prec$ the Bruhat order on it. 
For any simple reflection $t\in W$, we consider the minimal parabolic subgroup $P_t=B\cup BtB$.

Let $s=(s_1,\ldots,s_n)$ be a sequence of simple reflections. The {\it classical Bott-Samelson} variety 
for $s$ is the quotient variety
$$
P_{s_1}\times\cdots\times P_{s_n}/B^n
$$ 
with the following right action of $B^n$:
$$
(p_1,p_2,\ldots,p_n)(b_1,b_2,\ldots,b_n)=(p_1b_1,b_1^{-1}p_2b_2,\ldots,b_{n-1}^{-1}p_nb_n).
$$
It is convenient here to redefine Bott-Samelson varieties using only compact groups.
Let $C$, $K$ and $C_t$ be the maximal compact subgroups of $G$, $T$ and $P_t$, respectively.  
The {\it compactly defined Bott-Samelson variety} for $s$ is the quotient variety
$$
C_{s_1}\times\cdots\times C_{s_n}/K^n
$$
under the similar right action of $K^n$:
$$
(c_1,c_2,\ldots,c_n)(k_1,k_2,\ldots,k_n)=(c_1k_1,k_1^{-1}c_2k_2,\ldots,k_{n-1}^{-1}c_nk_n).
$$
It is well-known~\cite{GK} that both varieties defined above are homeomorphic 
(for the same sequence of simple reflections). We will denote them by $\BS(s)$, call them simply the {\it Bott-Samelson} variety 
and always use only the definition via the compact groups.

\subsection{Combinatorial galleries}\label{Combinatorial galleries} Let $\Phi$ be the root system of $G$, 
$\Phi^+$ be the system of positive roots defined by the choice of $B$ and $A$ be the fundamental chamber of $\Phi^+$.
We define $\Phi^-=-\Phi^+$. For $\alpha\in\Phi^+$, we write $\alpha>0$ and call this root {\it positive}.
Similarly, for $\alpha\in\Phi^-$, we write $\alpha<0$ and call this root {\it negative}.
We say that roots $\alpha$ and $\beta$ {\it have the same sign} if either $\alpha,\beta>0$ or $\alpha,\beta<0$. 

Let us consider a sequence of simple reflections $s=(s_1,\ldots,s_n)$. 
The elements of the set
$$
\Gal(s)=\{(\gamma_1,\ldots,\gamma_n)\in W^n\suchthat\gamma_i=1\text{ or }\gamma_i=s_i\}.
$$
are called {\it combinatorial galleries}.
We will use the following notation:
$$
\gamma^i=\gamma_1\gamma_2\cdots\gamma_i,\qquad
\bbeta_i(\gamma)=\gamma^i(-\alpha_i),\quad
\widetilde\bbeta_i(\gamma)=\gamma^{i-1}(-\alpha_i),
$$
where $\gamma\in\Gal(s)$ and $s_i=\omega_{\alpha_i}$.
There is the following one-to-one correspondence between $\bigcup_s\Gal(s)$
and the set of all galleries in $\Phi$ starting at $A$:
\begin{equation}\label{eq:8}
\gamma=(\gamma_1,\ldots,\gamma_n)\mapsto \gal(\gamma)=(A,L_{\bbeta_1(\gamma)},\gamma^1A,L_{\bbeta_2(\gamma)},\gamma^2A,\ldots,L_{\bbeta_n(\gamma)},\gamma^nA).
\end{equation}
To specify the final chamber of this gallery, we denote $\ppi(\gamma)=\gamma^n$.
We denote $|\gamma|=n$ and call this number the {\it length} of $\gamma$.
Clearly $|\gamma|=|\gal(\gamma)|$.

There is the {\it folding operator} $\f_i$
on $\Gal(s)$ that multiplies the $i$th element of a combinatorial gallery by $s_i$.
We have
\begin{equation}\label{eq:v}
\bbeta_i(\f_j\gamma)=\left\{\begin{array}{ll}\bbeta_i(\gamma)&\text{ if }i<j;\\\omega_{\bbeta_j(\gamma)}\bbeta_i(\gamma)&\text{ if }i\ge j\end{array}\right.,
\qquad
\tilde\bbeta_i(\f_j\gamma)=\left\{\begin{array}{ll}\widetilde\bbeta_i(\gamma)&\text{ if }i\le j;\\\omega_{\widetilde\bbeta_j(\gamma)}\widetilde\bbeta_i(\gamma)&\text{ if }i> j.\end{array}\right.
\end{equation}

There is also the {\it double folding operator} $\f_{i,j}$, where $i<j$, that is just equal to the composition $\f_i\f_j$
and is applicable to a combinatorial gallery $\gamma$ if and only if $\bbeta_i(\gamma)=\pm\bbeta_j(\gamma)$.
This operator has the property
\begin{equation}\label{eq:vii}
\ppi(\f_{i,j}(\gamma))=\ppi(\gamma),
\end{equation}
whenever it is applicable to $\gamma$.

The set $\Gal(s)$ is identified with the set $\BS(s)^K$ of the $K$-fixed points of $\BS(s)$ by $\gamma\mapsto\gamma K^n$.
For any $K$-equivariant map $\phi:X\to Y$ between sets acted upon by $K$, we denote by $\phi^K:X^K\to Y^K$
the restriction of $\phi$ to the sets of $K$-fixed points.

\subsection{Orders}\label{Orders} Following~\cite{Haerterich}, we introduce the strict partial orders $\lhd$ and $<$ on $\Gal(s)$ by
{\renewcommand{\labelenumi}{{\rm\theenumi}}
\renewcommand{\theenumi}{{\rm(\roman{enumi})}}
\begin{enumerate}\itemindent=-5pt
\item $\delta\lhd\gamma$ if and only if $\delta^i\prec\gamma^i$ for the minimal index $i$ such that $\delta^i\ne\gamma^i$.\\[-8pt]
\item $\delta<\gamma$ if and only if $\ppi(\delta)\,{=}\,\ppi(\gamma)$ and $\delta^i\prec\gamma^i$ for the maximal index $i$ such that $\delta^i\ne\gamma^i$.
\end{enumerate}}
\noindent
This definitions allow reformulations in terms of the operators $\bbeta_i$ and $\tilde\bbeta_i$ introduced above.
{\renewcommand{\labelenumi}{{\rm\theenumi}}
\renewcommand{\theenumi}{{\rm(\roman{enumi}${}'$)}}
\begin{enumerate}\itemindent=-12pt
\item\label{lhd}{} $\delta\lhd\gamma$ if and only if $\bbeta_i(\gamma)>0$ for the minimal index $i$ such that $\delta_i\ne\gamma_i$.\\[-8pt]
\item\label{<} $\delta\,{<}\,\gamma$ if and only if $\ppi(\delta)\,{=}\,\ppi(\gamma)$ and $\widetilde\bbeta_i(\gamma)\;{>}\;0$ for the maximal index $i$ such that $\delta_i\;{\ne}\;\gamma_i$.
\end{enumerate}}

\smallskip

\noindent
We remember here the following simple fact that we will need in the proof of Theorem~\ref{theorem:4}.

\begin{proposition}\label{proposition:5}
Let $x,y\in W$ and $\alpha\in\Phi$. The hyperplane $L_\alpha$ does not separate the chambers $xA$ and $yA$ if and only if 
the roots $x^{-1}\alpha$ and $y^{-1}\alpha$ 
have the same sign.
\end{proposition}


\subsection{$(p,w)$-pairs}\label{(p,w)-pairs} Let $s=(s_1,\ldots,s_m)$ and $r=(r_1,\ldots,r_n)$ be sequences of simple reflections,
$p:[1,m]\to[1,n]$ be an increasing embedding and $w\in W$. A pair of galleries $(\gamma,\delta)\in\Gal(s)\times\Gal(r)$
is called a {\it $(p,w)$-pair} if
$$
\delta^{p(i)}r_{p(i)}(\delta^{p(i)})^{-1}=w\gamma^is_i(\gamma^i)^{-1}w^{-1}
$$
for all $i=1,\ldots,m$. We shall write these equations in either of more convenient forms
\begin{equation}\label{eq:iv}
\bbeta_{p(i)}(\delta)=\epsilon_iw\bbeta_i(\gamma),\quad
\widetilde\bbeta_{p(i)}(\delta)=\mu_iw\widetilde\bbeta_i(\gamma),
\end{equation}
where $\epsilon_i=\pm1$ and $\mu_i=\pm1$. The sequences $(\epsilon_1,\ldots,\epsilon_n)$
and $(\mu_1,\ldots,\mu_n)$ are called the {\it sign} and the {\it cosign}
of the $p$-pair $(\gamma,\delta)$.

There is the following connection between the notion of a $(p,w)$-pair and a $p$-pair
(see Definition~\ref{def:1}). We leave the proof of the following result to the reader.

\begin{proposition}\label{proposition:3} Let $\gamma\in\Gal(s)$ and $\delta\in\Gal(r)$,
$p:[1,|s|]\to[1,|r|]$ be an increasing embedding and $w\in W$.
Then $(\gamma,\delta)$ is a $(p,w)$-pair if and only if $(w\gal(\gamma),\gal(\delta))$ is a $p$-pair.
The signs and cosigns of these pairs coincide.
\end{proposition}

\subsection{Category $\widetilde\Seq$}\label{Category_widetilde_Seq}
The objects of this category are the sequences of simple reflections including the empty one.
Each morphism $s=(s_1,\ldots,s_m)\to r=(r_1,\ldots,r_n)$ is a triple $(p,w,\phi)$ such that

\begin{enumerate}
\itemsep=3pt
\item\label{fcat:1} $p:[1,m]\to[1,n]$ is an increasing embedding;
\item\label{fcat:2}  $w\in W$;
\item\label{fcat:3}  $\phi:\Gal(s)\to\Gal(r)$ is a map such that $(\gamma,\phi(\gamma))$
                     is a $(p,w)$-pair for any $\gamma\in\Gal(s)$ and
                     $\phi(\f_i\gamma)=\f_{p(i)}\phi(\gamma)$ for any $\gamma\in\Gal(s)$ and $i=1,\ldots,m$.
\end{enumerate}
By~\cite[Lemma~3]{fcat}, it suffices that $(\gamma,\phi(\gamma))$ be a $(p,w)$-pair
just for one $\gamma\in\Gal(s)$. The (co)sign of $(p,w,\phi)$
is defined to be the (co)sign of each pair $(\gamma,\phi(\gamma))$.
By~\cite[Lemma~2]{fcat} and the similar argument for $\widetilde\bbeta_k$
this definition does not depend on the choice of $\gamma$.
The composition law for the triples is $(p',w',\phi')(p,w,\phi)=(p'p,w'w,\phi'\phi)$.


\subsection{Embeddings} The category $\widetilde\Seq$ yields the following result on embeddings. 
\begin{proposition}[\mbox{\cite[Lemma 13]{bt}}]\label{proposition:4}
For any morphism $(p,w,\phi):s\to r$
of the category $\widetilde\Seq$, there exists a topological embedding $\iota:\BS(s)\to\BS(r)$
such that $\iota^K=\phi$ and $\iota(ka)=wkw^{-1}a$ for any $k\in K$ and $a\in\BS(s)$.
\end{proposition}

\noindent
Constructing $p$-pairs with the help of Theorem~\ref{theorem:3}, we ge the following result. 

\begin{theorem}\label{theorem:4}
Let $r$ be a sequence of simple reflections, $\gamma\in\Gal(r)$ and $\Psi$ be a saturated root subsystem of $\Phi$.
Let $\Delta$ be any gallery lifting $\gal(\gamma)_\Psi$, $w\in W$ be such that $w^{-1}\Delta$ begins at $A$ and
$s$ be the sequence of simple reflections such that $\gal^{-1}(w^{-1}\Delta)\in\Gal(s)$. 
Then there exists a topological embedding $\iota:\BS(s)\to\BS(r)$
mapping $\gal^{-1}(w^{-1}\Delta)$ to $\gamma$ such that $\iota(ka)=wkw^{-1}a$ for any $k\in K$, $a\in\BS(s)$ 
and $\iota^K$ preserves both orders $\lhd$ and $<$.
\end{theorem}
\begin{proof} We denote $\mu=\gal^{-1}(w^{-1}\Delta)$ for brevity and consider the map $p:[1,|s|]\to[1,|r|]$, 
which is the increasing imbedding with image $I_\Psi(\gal(\gamma))$. By Theorem~\ref{theorem:3}, 
the pair $(\Delta,\gal(\gamma))$ is a $p$-pair of positive sigh and cosign. As
$$
(\Delta,\gal(\gamma))=(ww^{-1}\Delta,\gal(\gamma))=(w\gal(\mu),\gal(\gamma)),
$$
we obtain from Proposition~\ref{proposition:3} that $(\mu,\gamma)$ is a $(p,w)$-pair of positive sign and cosign.
By~\cite[Lemma 3]{fcat}, we obtain a map $\phi:\Gal(s)\to\Gal(r)$ such that $\phi(\mu)=\gamma$ 
and $(p,w,\phi):s\to r$ is a morphism in $\widetilde\Seq$. Applying Proposition~\ref{proposition:4},
we obtain the required map $\iota$. 

It remains to prove that $\phi=\iota_K$ preserves both orders $\lhd$ and $<$.
Let $\lambda,\rho\in\Gal(s)$ be such that $\lambda\lhd\rho$. Let $i$ be the minimal index such that $\lambda_i\ne\rho_i$.
Then there exist indices $i_1,\ldots,i_m>i$ such that $\rho=\f_i\f_{i_1}\cdots \f_{i_k}\lambda$.
Applying property~\ref{fcat:3} of the definition of $\widetilde\Seq$, we get
$\phi(\rho)=\f_{p(i)}\f_{p(i_1)}\cdots \f_{p(i_k)}\phi(\lambda)$. As $p$ is increasing, we get 
$p(i_1),\ldots,p(i_k)>p(i)$, which proves that $\phi(\lambda)_{p(i)}\ne\phi(\rho)_{p(i)}$ but
$\phi(\lambda)_j=\phi(\rho)_j$ for any $1\le j<p(i)$. 

Remember that $p(j)\in I_\Psi(\gal(\gamma))$ for any $j\in[1,|s|]$ by the definition of $p$. Hence from~(\ref{eq:8}),
we get 
\begin{equation}\label{eq:vi}
\bbeta_{p(j)}(\gamma)\in\Psi.
\end{equation}
Now we consider a representation $\rho=\f_{j_m}\cdots\f_{j_1}\mu$.
Applying $\phi$ to both sides, we get by part~\ref{fcat:3} of the definition of $\widetilde\Seq$
that $\phi(\rho)=\f_{p(j_m)}\cdots\f_{p(j_1)}\gamma$. We are going to prove by induction on $m$ 
that $\bbeta_{p(j)}(\phi(\rho))\in\Psi$ for any $j\in[1,|s|]$. The case $m=0$ follows from~(\ref{eq:vi}).
Suppose that $m>0$ and the claim holds for $m-1$ folding operators. By the first formula of~(\ref{eq:v}) 
and the inductive hypothesis, we get for some $\epsilon\in\{0,1\}$ that
$$
\bbeta_{p(j)}(\phi(\rho))=\bbeta_{p(j)}(\f_{p(j_m)}\cdots\f_{p(j_1)}\gamma)
=\omega_{\bbeta_{p(j_m)}(\f_{p(j_{m-1})}\cdots\f_{p(j_1)}\gamma)}^\epsilon \bbeta_{p(j)}(\f_{p(j_{m-1})}\cdots\f_{p(j_1)}\gamma)\in\Psi.
$$

By the first formula of~(\ref{eq:iv})
and the positivity of the sign, we get $\bbeta_{p(i)}(\phi(\rho))=w\bbeta_i(\rho)$.
As $\Delta$ lifts $\gal(\gamma)_\Psi$, the initial chamber $wA$ of $\Delta$ lifts the initial chamber $A_\Psi$ of $\gal(\gamma)_\Psi$.
Hence by Lemma~\ref{lemma:1}, we get $(wA)_\Psi=A_\Psi$, whence no wall $L_\beta$ for $\beta\in\Psi$ separates $A$ and $wA$. 
In particular the wall $L_{\bbeta_{p(i)}(\phi(\rho))}$ does not separate these chambers. 
By Proposition~\ref{proposition:5}, we get that the roots $\bbeta_{p(i)}(\phi(\rho))$ and 
$w^{-1}\bbeta_{p(i)}(\phi(\rho))=\bbeta_i(\rho)$ have the same sign. However, as $\lambda\lhd\rho$, 
the latter root is positive by~\ref{lhd} from Section~\ref{Orders}. 
Hence $\bbeta_{p(i)}(\phi(\rho))>0$ and $\phi(\lm)\lhd\phi(\rho)$ again by~\ref{lhd}.

The proof of the fact that $\iota^K$ preserves $<$ is very similar. Indeed, let 
$\lambda,\rho\in\Gal(s)$ be such that $\lambda<\rho$. Let $i$ be the maximal index such that $\lambda_i\ne\rho_i$.
Similarly to the above arguments, we get $\phi(\lambda)_j=\phi(\rho)_j$ for any $p(i)<j\le|s|$.
As $\ppi(\lambda)=\ppi(\rho)$, we get $\rho=\f_{i_1,j_1}\cdots\f_{i_k,j_k}\lambda$ 
for some sequence of the double folding operators. Applying property~\ref{fcat:3} of the definition of $\widetilde\Seq$, we get
$\phi(\rho)=\f_{p(i_1),p(j_1)}\cdots \f_{p(i_k),p(j_k)}\phi(\lambda)$. Hence $\ppi(\phi(\rho))=\ppi(\phi(\lambda))$
by~(\ref{eq:vii}).

Just as for the case of $\lhd$, we can prove that $\widetilde\bbeta_{p(i)}(\phi(\rho))\in\Psi$. 
By the second formula of~(\ref{eq:iv})
and the positivity of the cosign, we get $\widetilde\bbeta_{p(i)}(\phi(\rho))=w\widetilde\bbeta_i(\rho)$.
As $L_{\tilde\bbeta_{p(i)}(\phi(\rho))}=L_{\bbeta_{p(i)}(\phi(\rho))}$ and we proved above that 
this wall does not separate $A$ and $wA$, we get by Proposition~\ref{proposition:5} that
the roots $\widetilde\bbeta_{p(i)}(\phi(\rho))$ and $w^{-1}\widetilde\bbeta_{p(i)}(\phi(\rho))=\widetilde\bbeta_i(\rho)$ 
have the same sign. However, as $\lambda<\rho$, the latter root is positive by~\ref{<} from Section~\ref{Orders}. 
Hence $\bbeta_{p(i)}(\phi(\rho))>0$ and $\phi(\lm)<\phi(\rho)$ again by~\ref{<}.
\end{proof}

\begin{remark}\rm Proposition~\ref{proposition:4} allows us to construct embeddings of Bott-Samelson varieties
for lifings $\Delta$ different from that used in Theorem~\ref{theorem:4}, for example, for several such liftings 
glued together as in Example 2 below. Wether the orders $\lhd$ and $<$ are preserved is an open question. 
\end{remark}

\subsection{Example 2} Let $\Phi$ be the root system of type $A_3$ with
the following Dynkin diagram:
\begin{center}
\setlength{\unitlength}{1.2mm}
\begin{picture}(45,0)
\put(0,0){\circle{1}}
\put(10,0){\circle{1}}
\put(20,0){\circle{1}}

\put(0.5,0){\line(1,0){9}}
\put(10.5,0){\line(1,0){9}}

\put(-1,-3.5){$\scriptstyle\alpha_1$}
\put(9,-3.5){$\scriptstyle\alpha_2$}
\put(19,-3.5){$\scriptstyle\alpha_3$}

\end{picture}
\end{center}

\hspace{20mm}

\noindent
The simple reflections are $\omega_i=\omega_{\alpha_i}$. Let us take
$r=(\omega_2,\omega_3,\omega_2,\omega_1,\omega_2, \omega_1,\omega_3,\omega_1)$
and $\gamma=(1       ,\omega_3,       1,\omega_1,\omega_2, \omega_1,1,\omega_1)\in\Gal(r)$ 
and consider the gallery $\Gamma=\gal(\gamma)$. We draw $\Gamma$ indicating the segments
where the root subsystems $\Psi_1$ and $\Psi_2$ are applied as follows:
\begin{multline*}
\Gamma=
(\overbrace{A,L_{\alpha_2},A,L_{\alpha_3},\omega_3A,L_{\alpha_2+\alpha_3},\omega_3A,L_{\alpha_1},\omega_3\omega_1A,L_{\alpha_1+\alpha_2+\alpha_3},\omega_3\omega_1\omega_2A}^{\Psi_1},\!\!\!\!\!\!\!\!\!\!\!\!\!\!\!\!\!\!\!\!\!\!\!\!\!\!\underbrace{\qquad\qquad}_{\Psi_2}\\
 \underbrace{L_{\alpha_2+\alpha_3},\omega_3\omega_1\omega_2\omega_1A,L_{\alpha_1+\alpha_2},\omega_3\omega_1\omega_2\omega_1A,L_{\alpha_2+\alpha_3},\omega_3\omega_1\omega_2A}_{\Psi_2}).
\end{multline*}
Here as usual $A$ denotes the fundamental chamber. We are going to project, lift and glue according to the general plan 
described in Section~\ref{Gluing}. During this process, we will choose the root systems $\Psi_1$ and $\Psi_2$.

First, we take $\Psi_1=\R\{\alpha_1+\alpha_2,\alpha_3\}\cap\Phi$. Then we get $\Psi_1^s(A_{\Psi_1})=\{\alpha_1+\alpha_2,\alpha_3\}$
and
$$
(\Gamma_{[0,5]})_{\Psi_1}=(A_{\Psi_1},L_{\alpha_3},\omega_3A_{\Psi_1},L_{\alpha_1+\alpha_2+\alpha_3},\omega_3\omega_2A_{\Psi_1}).
$$
Now we are going to lift this gallery. In order to do it, we need lift
its initial chamber $A_{\Psi_1}$. There are two possible liftings $\omega_1A$
and $\omega_2\omega_3A$. Let us take the first one.
We get
$$
{}^{\omega_1A}(\Gamma_{[0,5]})_{\Psi_1}=(\omega_1A,L_{\alpha_3},\omega_1\omega_3A,L_{\alpha_1+\alpha_2+\alpha_3},\omega_1\omega_3\omega_2A).
$$
We have $\Phi^s(\omega_1\omega_3\omega_2A)=\{\alpha_2+\alpha_3, \alpha_1+\alpha_2, -\alpha_1-\alpha_2-\alpha_3\}$.
Thus we can take $\Psi_2=\R\{\alpha_2+\alpha_3\}\cap\Phi=\{\alpha_2+\alpha_3,-\alpha_2-\alpha_3\}$.
We get
$$
(\Gamma_{[5,8]})_{\Psi_2}=(\omega_1\omega_3\omega_2A_{\Psi_2},L_{\alpha_2+\alpha_3},\omega_1\omega_3\omega_2\omega_1A_{\Psi_2},L_{\alpha_2+\alpha_3},\omega_1\omega_3\omega_2A_{\Psi_2}).
$$
Then we lift
$$
{}^{\omega_1\omega_3\omega_2A}(\Gamma_{[5,8]})_{\Psi_2}=(\omega_1\omega_3\omega_2A,L_{\alpha_2+\alpha_3},\omega_1\omega_3\omega_2\omega_1A,L_{\alpha_2+\alpha_3},\omega_1\omega_3\omega_2A).
$$
Finally we glue both liftings and get the gallery
\begin{multline*}
\Delta={}^{\omega_1A}(\Gamma_{[0,5]})_{\Psi_1}\cup{}^{\omega_3\omega_2\omega_1A}(\Gamma_{[5,8]})_{\Psi_2}\\
=(\omega_1A,L_{\alpha_3},\omega_1\omega_3A,L_{\alpha_1+\alpha_2+\alpha_3},\omega_1\omega_3\omega_2A,L_{\alpha_2+\alpha_3},\omega_1\omega_3\omega_2\omega_1A,L_{\alpha_2+\alpha_3},\omega_1\omega_3\omega_2A).
\end{multline*}
Hence it is clear that $(\Delta,\Gamma)$ is a $p$-pair, where $p:[1,4]\to\{2,5,6,8\}$
is the increasing bijection.
Multiplying by $\omega_1$, we get the gallery starting at the fundamental chamber
$$
\omega_1\Delta=(A,L_{\alpha_3},\omega_3A,L_{\alpha_2+\alpha_3},\omega_3\omega_2A,L_{\alpha_1+\alpha_2+\alpha_3},\omega_3\omega_2\omega_1A,L_{\alpha_1+\alpha_2+\alpha_3},\omega_3\omega_2A).
$$
We get the combinatorial gallery $\delta=\gal^{-1}(\omega_1\Delta)=(\omega_3,\omega_2,\omega_1,\omega_1)\in\Gal(s)$ 
for $s=(\omega_3,\omega_2,\omega_1,\omega_1)$. The reader can easily check that
$(\delta,\gamma)$ is a $(p,\omega_1)$-pair of positive sign and cosign.
By Proposition~\ref{proposition:4}, we get an embedding
$$
\iota:\BS(\omega_3,\omega_2,\omega_1,\omega_1)\hookrightarrow\BS(\omega_2,\omega_3,\omega_2,\omega_1,\omega_2, \omega_1,\omega_3,\omega_1),
$$
such that $\iota(ka)=\omega_1k\omega_1^{-1}a$ for any $k\in K$ and $a\in\BS(s)$ and
$$
\iota(         \omega_3^{\epsilon_1},\omega_2^{\epsilon_2},\omega_1^{\epsilon_3},\omega_1^{\epsilon_4})
=   (       1,\omega_3^{\epsilon_1},                    1,\omega_1,\omega_2^{\epsilon_2}, \omega_1^{\epsilon_3},1,\omega_1^{\epsilon_4})
$$
for any $\epsilon_1,\epsilon_2,\epsilon_3,\epsilon_4\in\{0,1\}$. 
Here we used the fact that $(p,\omega_1,\iota^K)$ is a morphism in $\widetilde\Seq$.

\bigskip

\def\sep{\\[-7pt]}

\bigskip

\end{document}